\documentclass[reqno,12pt]{amsart}

\usepackage{psfig}

\usepackage{graphicx}
\usepackage{amssymb,amsmath,amstext}
\usepackage{latexsym}

\usepackage[norelsize]{algorithm2e}
\usepackage{color}
\usepackage{algorithmic}
\renewcommand{\algorithmcfname}{ALGORITHM}

\newtheorem{assumption}{Assumption}
\newcommand{\MNIST}{{\small MNIST}}
\newcommand{\EMNIST}{{\small EMNIST}}

\newtheorem{theorem}{Theorem}
\newtheorem{lemma}{Lemma}

\begin{document}

\title[Trust-Region Algorithms for Training Responses: 
Machine Learning Methods Using Indefinite Hessian Approximations]
{Trust-Region Algorithms for Training Responses: 
Machine Learning Methods Using Indefinite Hessian Approximations}


\author[J. Erway]{Jennifer B. Erway}
\email{erwayjb@wfu.edu}
\address{Department of Mathematics, Wake Forest University, Winston-Salem, NC 27109}

\author[J. Griffin]{Joshua Griffin}
\email{Joshua.Griffin@sas.com}
\address{SAS, Cary, NC 27513}

\author[R. Marcia]{Roummel F. Marcia}
\email{rmarcia@ucmerced.edu}
\address{Department of Applied Mathematics, University of California, Merced, Merced, CA 95343}

\author[R. Ohmeni]{Riadh Omheni}
\email{Riadh.Omheni@sas.com}
\address{SAS, Cary, NC 27513}

\thanks{J.~B. Erway is supported in part by National Science Foundation grants
CMMI-1334042 and IIS-1741264}
\thanks{R.~F. Marcia is supported in part by National Science Foundation grants
CMMI-1333326 and IIS-1741490}

\date{\today}

\keywords{Large-scale optimization, non-convex, machine learning,
  trust-region methods, quasi-Newton methods, 
  limited-memory symmetric rank-one update}

\begin{abstract}
Machine learning (ML) problems are often posed as highly nonlinear and
nonconvex unconstrained optimization problems.  Methods for solving ML
problems based on stochastic gradient descent are easily scaled for very
large problems but may involve fine-tuning many hyper-parameters.
Quasi-Newton approaches based on the limited-memory
Broyden-Fletcher-Goldfarb-Shanno (BFGS) update typically do not require
manually tuning hyper-parameters but suffer from approximating a
potentially indefinite Hessian with a positive-definite matrix.
Hessian-free methods leverage the ability to perform Hessian-vector
multiplication without needing the entire Hessian matrix, but each
iteration's complexity is significantly greater than quasi-Newton methods.
In this paper we propose an alternative approach for solving ML problems
based on a quasi-Newton trust-region framework for solving large-scale
optimization problems that allow for indefinite Hessian approximations.
Numerical experiments on a standard testing data set show that with a fixed
computational time budget, the proposed methods achieve better results than
the traditional limited-memory BFGS and the Hessian-free methods.
\end{abstract}

\maketitle

\newcommand{\mgap}{\;\;}
\newcommand{\bgap}{\;\;\;}
\newcommand{\qDef}{{\mathcal Q}}
\newcommand{\defined}{\mathop{\,{\scriptstyle\stackrel{\triangle}{=}}}\,}
\newcommand{\diag}{\text{diag}}
\renewcommand{\algorithmcfname}{ALGORITHM}

\makeatletter
\newcommand{\minimize}[1]{{\displaystyle\minim_{#1}}}
\newcommand{\minim}{\mathop{\operator@font{minimize}}}
\newcommand{\subject}{\mathop{\operator@font{subject\ to}}}  
\newcommand{\words}[1]{\mgap\text{#1}\mgap}
\def\BFGS{{\small BFGS}}
\def\LBFGS{{\small L-BFGS}}
\def\LSR{{\small L-SR1}}
\def\SR{{\small SR1}}
\def\CG{{\small CG}}
\def\DFP{{\small DFP}}
\def\OBS{{\small OBS}}
\def\OBSSC{{\small OBS-SC}}
\def\SCSR1{{\small SC-SR1}}
\def\PSB{{\small PSB}}
\def\QR{{\small QR}}
\def\MATLAB{{\small MATLAB}}
\def\SGD{{\small SGD}}
\newcommand{\MCG}{{\small MCG}}
\newcommand{\SSM}{{\small SSM}}
\newcommand{\LSTRS}{{\small LSTRS}}
\newcommand{\phasedSSM}{phased-\SSM}
\newcommand{\transpose}{T}
\providecommand{\GLTR   }{{\small GLTR}}
\def\LSR1{{\small L-SR1}}
\def\LSRTR{{\small L-SR1-TR}}
\def\SR1{{\small SR1}}
\def\LSRTR{{\small L-SR1-TR}}
\def\LSSRTR{{\small L-SSR1-TR}}
\newcommand{\Iset}{\mathcal{I}}

\renewcommand{\algorithmcfname}{ALGORITHM}
\renewcommand{\vec}[1]{#1}

\makeatother

\pagestyle{myheadings}
\thispagestyle{plain}
\markboth{J. B. ERWAY, J. GRIFFIN, R. F. MARCIA, and R. OMHENI}{Trust-Region Algorithms for 
Machine Learning  Using Indefinite Hessian Approximations}

\section{Introduction}
\label{sec:intro}
Machine learning problems, such as text classification and speech recognition,
are often nonlinear and
nonconvex unconstrained problems of the form 
\begin{equation}\label{eqn-min}
\min_{w\in\Re^m} f(w) \defined \frac{1}{n}\sum_{i=1}^n f_i(w),
\end{equation}
where $f_i$ is a function of the $i$th observation in a \emph{training}
data set $\{(x_i,y_i)\}$
with $x_i \in \Re^{d}$ and $y_i\in\Re^o$.
In the literature (see e.g., \cite{Vap92,Nocedal2016}), (\ref{eqn-min}) is often
referred to as the \emph{empirical risk}.  Generally speaking, these
problems have several features that make traditional optimization
algorithms ineffective.  First, both $m$ and $n$ are very large (e.g.,
typically $m,n\ge 10^6$).  Second, there is a special type of redundancy
that is present in (\ref{eqn-min}) due to similarity between data points;
namely, if $\mathcal{I}$ is a random subset of indices of $\{1,2,\ldots, n\}$, then
provided $\mathcal{I}$ is large enough, but $|\mathcal{I}|\ll n$ (e.g., $n=10^9$ and
$|\mathcal{I}|=10^5$), then
\begin{equation}\label{eqn-batch}
	\hat{f}(w) \defined
	\frac{1}{|\Iset |}\sum_{i\in \Iset} f_i(w) \approx
	\frac{1}{n}\sum_{i=1}^n f_i(w).	
\end{equation}

The underlying goal during the optimization phase in machine learning is to
find the ``best" set of model parameters $w$ so that the chosen model
function $p(x,w): \Re^d \times \Re^m \to \Re^o$ predicts the observed
target variable $y \in \Re^o$ as accurately as possible.  
The most popular approaches in machine learning include
(i) stochastic gradient
descent method, (ii) limited-memory \BFGS, and (iii) Hessian-free methods.  
Here we briefly describe each approach and describe their advantages and disadvantages.

\medskip

\noindent \textbf{(i) Stochastic gradient descent (\SGD) methods.}
The stochastic gradient descent (\SGD) method~\cite{Robbins1951}
is one of the most popular types of methods for solving machine learning problems.
For this iterative method, an index $j$ is randomly chosen from 
$\{1,2,\ldots n\}$ at each iteration and $w$ is updated as follows:
$$
w = w-\eta_j \nabla f_j(w),
$$
where $\nabla f_j$ denotes the gradient of $f_j$.  The parameter $\eta_j$
is referred to as the \emph{learning rate} in machine learning.  \SGD{} is
a very attractive algorithm for machine learning 
for several reasons.  First, it naturally exploits data set
redundancy described in (\ref{eqn-batch}); moreover, the iteration
complexity is independent of $n$.  In contrast, classical optimization
algorithms are explicitly dependent on $n$ and become much more unstable
when attempting exploit cheaper stochastic approximations of the
gradient~\cite{Nocedal2016,Byrd2016,Curtis20016,Curtis2015,Gower2016,
  Moritz2016,Schraudolph2007}.  Second, the algorithm
 comes with attractive convergence theory~\cite{Nocedal2016}.  Third,
 the \SGD{} algorithm readily responds to an on-line learning environment 
 (i.e., data is available in a sequential order instead of all-at-once)
where data observations may never repeat.
A fourth advantage occurs in the nonconvex setting
where the stochastic
nature of \SGD{} makes it much less likely to 
converge to inferior local minimums~\cite{Choromanska2014,Kawaguchi2016,Sagun2014}
than non-stochastic methods.

There are several important disadvantages associated with using
\SGD. To enhance the performance of \SGD{} in practice, developers
must fine-tune many hyper-parameters--leading to
many variations of \SGD{}, (e.g., see
~\cite{Abadi2016,Duchi2011,Ioffe2015,Kingma2014,Sutskever2013,Zeiler2012,
Zhang2015}).  One set of
hyper-parameter users must choose is a learning rate sequence (i.e.,
$\{\eta_j\}$).  If the learning rate is too small, the algorithm
may stall; on the other hand, if the learning rate is too large
the algorithm may not converge.
In practice, finding an effective sequence $\{\eta_j\}$ can require solving
the same problem many times to find the best sequence.  This dilemma has
led to a resurgence of interest in auto-tune algorithms that can aid the
\SGD{} user in this
search~\cite{Bengio2012,Bergstra2012,Bergstra2013,Bergstra2011,Dewancker2016,Le2011,Maclaurin2015,Snoek2012,Sparks2015}.
A second disadvantage with \SGD{} is that it is  inherently sequential, making it difficult to 
parallelize~\cite{Dean2012,Le2011,McMahan2014,Recht2011}.  

\bigskip

\noindent \textbf{(ii) Limited-memory \BFGS{} (\LBFGS{}).} 
One of the most popular classical algorithms in machine learning 
is the \LBFGS{} algorithm, which falls into
the class of limited memory quasi-Newton algorithms.  Like \SGD,
quasi-Newton algorithms require only first-order (gradient) information.
Quasi-Newton methods generate a sequence of iterates using the rule
\begin{equation}\label{eqn-qN}
	w_{k+1}
	= w_k +
\eta_k
	p_k,\quad \text{where}\quad p_k\defined -B_k^{-1}\nabla f(w_k),
\end{equation}
$B_k$ is a \emph{quasi-Newton} matrix that is updated at each iteration
using gradient information,
and $\eta_k$ 
is a suitably-defined step length (learning rate).
The update to $B_k$ is defined using 
sequences of vectors $\{s_j\}$ and $\{y_j\}$, which are given as
\begin{equation}\label{eq:sjyj}
	s_j\defined w_{j+1}-w_j \qquad \text{and} \qquad
	y_j\defined\nabla f(w_{j+1})-\nabla f(w_j),
\end{equation}
for $j=0,\ldots,k-1$.  The Broyden class of updates, parametrized
by $\phi\in\Re$, is the most widely-used
updating rule for $B_k$:
\begin{equation}\label{eqn-broyden}
 B_{k+1}  =B_k   -  \frac{1}{s_k ^TB_k  s_k }B_k  s_k s_k ^TB_k  +  
  	\frac{1}{y_k^Ts_k }y_ky_k^T +
         \phi (s_k^T B_k  s_k )v_kv_k^T,
        \end{equation}
where
\begin{equation*}
         v_k = \frac{y_k}{y_k^Ts_k} - \frac{B_k  s_k}{s_k^TB_k  s_k}.
\end{equation*}
In practice, $B_0$ is usually taken to be a positive scalar multiple of the
identity.  In large-scale optimization, \emph{limited-memory}
quasi-Newton methods are used to bound storage requirements and promote
efficiency.  In this case, only the $r$ most-recently computed pairs
$\{(s_j,y_j)\}$ are used to build $B_{k+1}$, i.e.,  only
 the most up-to-date information is used to model the Hessian matrix.
The value of  $r$ is typically very small so that $r\ll n$.

While the matrices in the sequence $\{B_j\}$ are symmetric by construction,
different choices of $\phi$ lead to sequences of matrices with
different properties.
The most well-known member of the Broyden class of updates is the
\emph{Broyden-Fletcher-Goldfarb-Shanno (\BFGS)} update, which is obtained
by setting $\phi=0$. 
Provided $B_0$ is positive definite and $y_i^Ts_i>0$ for each $i=0,\ldots,
k-1$, then the \BFGS{} update generates a sequence of symmetric
positive-definite matrices.  (The condition $y_i^Ts_i>0$ for each $i$ can
be enforced using a \emph{Wolfe line search} to compute $\eta_k$ in
(\ref{eqn-qN}).)  One reason why the \BFGS{} update is the
preferred update is that
there is an efficient way to solve linear systems with $B_k$,
making the computation of $p_k$ in (\ref{eqn-qN})
affordable~\cite{Noc80}.  
It is worth noting that of all the quasi-Newton updates available,
the limited-memory \BFGS{} (\LBFGS) update
has been used almost exclusively by researchers in machine learning.

\LBFGS{} has several advantages in the machine learning setting.
First, the computation of $\nabla f(w)$ benefits from a
parallel-programming environment.  Second, while there are only a few
hyper-parameters that the user may tune, such as 
the number of weights ($m$) used and the scaling
for the initial matrix $B_0$, there are known standard
initializations and values used by the optimization community; that is, \LBFGS{}
does not require manual tuning.

\LBFGS{} has a number of disadvantages for solving problems in 
machine learning, especially in deep learning, where
the network is composed of multiple cascading layers.
First, it cannot be used in an on-line learning
environment without significant modifications that limit its scalability to
arbitrarily large data sets.  (This has given rise to recent research into
\emph{stochastic} \LBFGS{} variations that have thus far been unable to
maintain the robustness of classical \LBFGS{} in a stochastic mini-batch
environment~\cite{Nocedal2016,Byrd2016,Curtis20016,Curtis2015,
Gower2016,Moritz2016,Schraudolph2007}.)  A third disadvantage of \LBFGS{}
occurs if one tries to enforce positive definiteness of the
\LBFGS{} matrices in a nonconvex setting.  In this case, \LBFGS{} has the
difficult task of approximating an indefinite matrix (the true Hessian)
with a positive-definite matrix $B_k$, which can result in the generation
of nearly-singular matrices $\{B_k\}$.  Numerically, this creates need
for heuristics such as periodically reinitializing $B_k$ to a multiple of
the identity, effectively generating a steepest-descent direction in the
next iteration.   This can be a significant disadvantage for 
neural network problems
where model quality is highly correlated with the quality of
initial steps~\cite{Ma12}. 

\bigskip

\noindent \textbf{(iii) Hessian-free (HF) methods.} 
A third family of algorithms of interest come from classical algorithms
that can leverage the ability to perform Hessian-vector multiplies without
needing the entire Hessian matrix
itself~\cite{dauphin14SFN,maGN,Ma11,Ma12}; for this reason, as
in~\cite{maGN,Ma12}, we will refer to this class as \emph{Hessian-free}
algorithms.  These algorithms perform approximate updates of the form
\begin{equation}\label{eqn-pd}
w_{k+1} = w_k + \eta_k p_k \quad \text{with} \quad \nabla^2 f(w_k) p_k = -\nabla f(w_k),
\end{equation}
where $p_k$ is an approximate Newton direction obtained computed using a
\emph{conjugate-gradient}-like (\CG-like) algorithm and $\eta_k$ is the step length.  Traditional \CG{}
algorithms assume $\nabla^2 f(w_k)$ is positive definite and solve for $p_k$ in
(\ref{eqn-pd}) using only matrix-vector products, and thus, are applicable
in large problems in machine learning. Because $\nabla^2 f(w_k)$ may be
indefinite in deep learning problems, modified variants are needed to adapt
for local nonconvexity; we refer to such approaches as \emph{modified
  conjugate-gradient} algorithms (\MCG). 

  Remarkably, Martens~\cite{maGN}, was able to show that
  Hessian-free methods were able to achieve out-of-the-box competitive
  results compared to manually-tuned \SGD{} on deep learning problems.
  Moreover,
  Pearlmutter~\cite{Pe94}  was able to show that matrix-vector products could
  be computed at a computational cost on the order of a gradient
  evaluation.  However, since multiple matrix-vector products can be
  required to solve (\ref{eqn-pd}), the iteration complexity of
  \MCG{} is significantly greater than \LBFGS.
Thus, despite its allure of
  being a tune-free approach to deep learning, Hessian-free methods are for the most part
  unused and unexplored in practice.
  
  \bigskip
  
\noindent \textbf{Contributions of the proposed method.}  
While the \BFGS{} update is the most widely-used type of quasi-Newton method
for general optimization as well as general machine learning,
it enjoys certain benefits (given by guaranteed positive-definite Hessian approximations)
that may actually hinder it in solving large nonconvex optimization problems.
Our proposed approach is based on a different quasi-Newton update,
namely the  \emph{symmetric rank-1} (\SR1{}) update,
which allows for indefinite Hessian approximation.  We use a trust-region framework
(see e.g.,  \cite{ConGT00a})
because this framework can accommodate indefinite Hessian approximations 
more easily (see \cite{NocW06}).  
We also present a stochastic extension of our proposed approach,
which improves computational time because it does not 
compute the full gradient at each iteration.

\section{L-SR1 trust-region methods}
\label{sec:background}
We begin by discussing the \SR1{} 
update and trust-region methods for large-scale optimization.

\subsection{The SR1 update}\label{sec-sr1}

The 
\SR1{} update is the
unique rank-one update in the Broyden class satisfying the so-called \emph{secant condition:
$$
	B_{k+1}s_k = y_k.
$$
}
This update occurs by setting
$\phi = y_k^Ts_k/(y_k^Ts_k-s_k^TB_ks_k)$ in (\ref{eqn-broyden}); in this case,
\begin{eqnarray}\label{eqn-SR1}
	B_{k+1} &=& B_k  + \frac{1}
	{s_k^T(y_k  - B_k s_k ) }(y_k -B_k s_k )(y_k  - B_k s_k )^T,
\end{eqnarray}
where $s_k$ and $y_k$ are defined in \eqref{eq:sjyj}.
At each iteration, we assume 
$(y_k - B_ks_k)^Ts_k \ne 0$, i.e., all of the updates are well-defined;
the update is skipped otherwise  (see \cite[Sec.\ 6.2]{NocW06}).
This update has the distinction of being the only rank-one update in the
Broyden class of updates.  Moreover, this update is self-dual: The
recursion (\ref{eqn-SR1}) can be used to generate $B_{k+1}^{-1}$ by
interchanging $y_k$ and $s_k$ everywhere in (\ref{eqn-SR1}) and
initializing with $B_0^{-1}$.  
Thus, linear systems with SR1 matrices can be solved efficiently.
An important aspect of the \SR1{} update
is that regardless of the sign of $y_i^Ts_i$ for each $i$, this update
generates a sequence of matrices that may be indefinite. It is precisely
this property of \SR1{} matrices that makes them attractive in applications
like deep learning where $f$ is nonconvex.

Decreasing the index by 1 in \eqref{eqn-SR1}, the \SR1{} update 
can be written recursively and compactly using the outer product representation
\begin{equation}\label{eqn-compactSR1}
	B_{k} \ = \ B_0 + 
	\begin{bmatrix}
	\\
	\Psi_k  \\
	\phantom{t}
	\end{bmatrix}
	\hspace{-.3cm}
	\begin{array}{c}
	\left  [ \ M_k^{\phantom{h}}  \right ] \\
	\\
	\\
	\end{array}
	\hspace{-.3cm}
	\begin{array}{c}
	\left [  \ \quad \Psi_k^{\transpose} \quad \ \right ] \\
	\\
	\\
	\end{array}
\end{equation}
where $B_0 = \gamma I$ for some $\gamma \ne 0$, 
 $\Psi_k$ is an $n \times k$ matrix and
 $M_k$ is a $k \times k$ matrix.
In the literature,  (\ref{eqn-compactSR1}) is referred to as the \emph{compact formulation}
of an \SR1{} matrix.
 In particular, Byrd et al.~\cite{ByrNS94} show that for \SR1{} matrices,
\begin{equation}\label{eq:PsiM}
	\Psi_k = Y_k\!  -\! B_0 S_k  \ \ \ \text{and} \ \  \
        M_k = (D_k \!+\! L_k \!+\! L_k^{\transpose} \!-\! S_k^{\transpose}\!B_0S_k)^{-1}\!.
\end{equation}
with
$S_k \defined [ \ s_0 \ \ s_1 \ \ s_2 \ \ \cdots \ \ s_{k-1} \ ] \ \in \ \Re^{n \times k},$ and
$Y_k \defined [ \ y_0 \ \ y_1 \ \ y_2 \ \ \cdots \ \ y_{k-1} \ ] \ \in \ \Re^{n \times k},$
and $L_k$ is the strictly lower triangular part, $U_k$ is the strictly
upper triangular part, and $D_k$ is the diagonal part of 
$
	S_k^{\transpose}Y_k =   L_k + D_k + U_k.
$
In our proposed approach, we use a limited-memory \SR1 (\LSR1{}) update,
where only $r$ of the most recent pairs $\{ s_j,y_j\}$ are stored,
where the value of $r$ is typically very small so that $r\ll n$.

While \SR1{} updates are one of many updates proven to theoretically
converge to the Hessian matrix at a minimizer, there is some
evidence that in practice \SR1{} updates 
 have superior convergence properties~\cite{ConGT91}.



\medskip

\noindent \textbf{The SR1 advantage:} Historically, the \SR1{} update fell
out of favor when it appeared to suffer from more algorithmic breakdowns
and instabilities than the \BFGS{} update; however, simple safeguards are
now used to adequately prevent instabilities and
breakdowns \cite[p.145]{NocW06}.  Over the last several decades, the \SR1{}
update has reemerged as the subject of much research; in fact, in \cite[p.118]{Gou06},
Gould states: ``[SR1] has now taken its place alongside the 
\BFGS{} method as the pre-eminent updating formula'' 

For machine learning, the \SR1{} update offers distinct advantages over the
\BFGS{} update: (i) In machine learning problems, Wolfe line searches to
enforce $y_k^Ts_k > 0$ in \BFGS{} methods
are too computationally expensive to use which has led to
the popular solution of skipping \BFGS{} updates, possibly
degrading the quality of the Hessian
approximation~\cite[p.146]{NocW06}; (ii) \SR1{} matrices exhibit better
convergence to the true Hessian (e.g., see the discussion on convergence in
Section~\ref{sec-sr1}); and (iii) if one tries to generate a sequence
of positive-definite \LBFGS{} matrices when modeling
 an indefinite Hessian, the matrices in this sequence 
 may become nearly
singular (i.e., highly ill-conditioned) with the smallest eigenvalue of
this sequence of matrices becoming close to zero.
Since machine learning problems are nonconvex,
it is worth noting that (ii) and (iii) suggest that
\SR1{} matrices may generate more accurate approximations than
positive-definite \LBFGS{} matrices of the true
Hessian.
Moreover, when (iii) occurs, the search direction obtained from a 
 \BFGS{}
method may be of poor quality, hindering convergence of the overall method.
In fact, research on \SR1{} methods have produced comparable, if not
better, results to \BFGS{} methods~\cite{ConGT88,ConGT91} on general
optimization problems.

\vspace{-.15cm}

\subsection{Large-scale trust-region methods}
Trust-region methods minimize a function $f$
by modeling changes in the objective function using quadratic models.  Each
iteration requires \emph{approximately} solving a \emph{trust-region
subproblem}.  Specifically, at the $k$th iteration, the $k$th
trust-region subproblem is given by
\begin{equation} \label{eq:TRsubproblem}
    p^* \ = \  \underset{p\in\Re^n}{\text{argmin}} \quad  \mathcal{Q}_k(p) \defined g_k^T p + \frac{1}{2} p^T B_k p
	\qquad
	\text{subject to} \quad \|p\| \le \delta_k,
\end{equation}
where $g_k \defined \nabla f(w_k)$, $B_k \approx \nabla^2 f(w_k)$, and
$\delta_k$ is a given positive \emph{trust-region radius}.
Basic trust-region methods update
the current approximate minimizer for $f$ only if the ratio
between the actual and predicted change in function value is
sufficiently large.  
If the ratio is sufficiently large, the update is accepted and
$w_{k+1}\gets w_k+p^*$.
When this
is not the case, $\delta_k$ is reduced and the trust-region subproblem is
resolved.  
The solution of the trust-region subproblem is the computational
bottleneck of most trust-region methods.
The primary advantage of using a trust-region method is that
$B_k$ does not have to be a positive-definite matrix; in particular,
it may be a limited-memory \SR1{} matrix.

Trust-region methods for general large scale optimization use an iterative
method to solve the trust-region subproblem.  It is well known that when
the two-norm is used to define the subproblem~(\ref{eq:TRsubproblem}), we
can completely characterize a global solution of the subproblem.  The
optimality conditions for the trust-region subproblem defined using the
two-norm are due to Gay~\cite{Gay81} and 
Mor\'e and Sorensen \cite{MorS84}:

\medskip

\noindent\textbf{Theorem:} \textsl{Let $\delta$ be a given positive constant. A vector $p^*$ is a
  global solution of the trust-region problem
  {\rm(\ref{eq:TRsubproblem})} if and only if $\|p^*\|_2 \le
  \delta$ and there exists a unique $\sigma^*\ge0$ such that
  $B_k+\sigma^* I$ is positive semidefinite with
\begin{equation}                              \label{eq:optimality}
  (B_k+\sigma^* I)p^* = - g_k,  \quad \text{and} \quad
   \sigma^*(\delta_k - \|p^*\|_2)=0.
\end{equation}
Moreover, if $B_k+\sigma^* I $ is positive definite, then the global
minimizer is unique.}

\medskip

\medskip

\noindent
Most iterative methods for solving the trust-region subproblem
assume it is possible to compute matrix-vector products with the
true Hessian, but matrix factorizations are too computationally
expensive to perform.  Examples of such methods include
Steihaug's method~\cite{Ste83}, Toint's method~\cite{Toi81},
the \GLTR{} method~\cite{GouLRT99}, \phasedSSM{}~\cite{ErwGG09},
Hager's \SSM{} method~\cite{Hag01},
Erway and Gill's \SSM{} method~\cite{ErwG09}, and the \LSTRS{} method~\cite{rojas2001, Rojas2008}.  In 
many machine learning applications, these methods
are too computationally expensive for use on the full data set.

\subsection{Solving the L-SR1 trust-region subproblem}\label{section:SolveLSR1}
Solving the trust-region subproblem \eqref{eq:TRsubproblem} 
is generally the computational bottleneck of trust-region methods.
In recent work by the authors \cite{BruEM17},
an efficient algorithm for solving the trust-region subproblem \eqref{eq:TRsubproblem} is proposed, where
$B_k$ is the SR1 quasi-Newton update. 
To efficiently solve the subproblems, we 
exploit the structure of the \LSR1 matrix
to obtain global solutions to high accuracy.
 We summarize this approach here.

To begin, we transform the optimality equations (\ref{eq:optimality})
using the spectral decomposition of $B_k$, which we outline here (see~\cite{BruEM17} for more details).
Given the compact formulation of $B_k$,  $B_{k}=B_0+\Psi M \Psi^{\transpose}$, and 
the ``thin'' \QR{} factorization
of $\Psi$, $\Psi=QR$,
then
$
	B_{k}=\gamma I + QRMR^{\transpose}Q^{\transpose},
$
where $B_0=\gamma I$ and $\gamma>0$ (see \cite{Burdakov2016,ErwM15}).  Since $RMR^{\transpose}$ is a small 
$k \times k$ matrix, its spectral decomposition $V\hat{\Lambda}V^{\transpose}$ 
can be quickly computed.  Then, letting $\Pi \defined [ \
QV \ \ \ (QV)^{\perp} ] \in \Re^{n \times n}$ such that $\Pi^{\transpose}\Pi = \Pi \ \! \Pi^{\transpose}\! =
I$, the spectral decomposition of $B_k$ is given by 
\begin{equation}\label{eqn-Beig}
	\phantom{.} \hspace{-.3cm} B_k = \Pi\Lambda \Pi^{\transpose}, \ 
	\text{where }
	\Lambda  \defined
	\begin{bmatrix}
		\Lambda_1 & 0 \\
		0 & \Lambda_2
	\end{bmatrix} = 
	\begin{bmatrix}
		\hat{\Lambda} + \gamma I & 0 \\
		0 & \gamma I
	\end{bmatrix},
\end{equation}
where $\Lambda_1$ $=$ diag(
$\hat{\lambda}_1+\gamma, \hat{\lambda}_2+\gamma, \dots, 
\hat{\lambda}_k+\gamma$)
$\in \Re^{k \times k},$ and $\Lambda_2 = \gamma
I_{n-k}$.  Using the spectral decomposition of $B_k$, the optimality equations
(\ref{eq:optimality}) become
\begin{subequations}
\begin{align}
	\label{eq:newopta}
	(\Lambda + \sigma^* I)v^* & = -\Pi^{\transpose}\!g \\
	\label{eq:newoptb}
	\sigma^* ( \delta_k - \| v^* \|_2) & = 0, 
\end{align}
\end{subequations}
for some scalar $\sigma^*\ge 0$ and $v^*= \Pi^{\transpose}p^*$, where $p^*$ is the global
solution to \eqref{eq:TRsubproblem}.   The Lagrange multiplier $\sigma^*$
can be obtained by substituting the expression
\begin{equation}\label{eq:normv}
	\| v^* \|_2^2 = \|(\Lambda+\sigma^*I)^{-1}\Pi^{\transpose}\!g\|_2^2
	=
	\sum_{i=1}^k
	\frac{(\Pi^{\transpose}\!g)_i^2}{(\hat{\lambda}_i+\gamma}+\sigma^*)^2
	+
	\frac{ \| \left ((QV)^{\perp} \right )^{T}\!g \|_2^2}{(\gamma + \sigma^*)^2}
\end{equation}
from \eqref{eq:newopta} into \eqref{eq:newoptb}
and finding the largest solution $\sigma^*$ to the
secular equation
\begin{equation}\label{eq:secular}
	\phi(\sigma) = \frac{1}{\| v(\sigma) \|_2} - \frac{1}{\delta} = 0
\end{equation}
using Newton's method.
Once $\sigma^*$ is obtained, $v^*$ can be computed from 
\eqref{eq:newopta} and as well as the solution $p^* = \Pi v^*$ to the original trust-region
subproblem \eqref{eq:TRsubproblem}. 
Note that in only one special case, the so-called \emph{hard case}~\cite{ConGT00a,MorS83}, the above method will not work because the
computed
$\|p^*\|$ will not lie on the boundary of the trust region.  In this case, the
global solution to the trust-region subproblem is given by 
${p}^* = \hat{p}^*+\alpha u_{\text{min}}$, 
where $\hat{p}^* = -(B_k+\sigma^*I)^{\dagger}g$,
$u_{\text{min}}$ is a column of $P$ and is
an eigenvector associated with the most negative eigenvalue of $B_k$ and can be
computed from the partial spectral decomposition outlined above, and 
$\alpha = \pm \sqrt{\delta^2 - \|\hat{p}^*\|^2}$ is a scalar
to ensure that $p^*$ lies on the boundary.
(See~\cite{BruEM17} for details on the hard case.)

\subsection{Proposed approach}
The proposed \LSR1 Trust-Region Method ({\small L-SR1-TR})
is outlined in Algorithm 1, and the trust-region subproblem
solver is described in Algorithm 2.  For details on the subproblem
solver and all related computations, see~\cite[Algorithm 1]{BruEM17}.

\begin{algorithm}[htp]
\caption{L-SR1 Trust-Region (L-SR1-TR) Method} 
\label{alg}
\begin{algorithmic}[1]
\REQUIRE $x_0\in \Re^n$, \ $\delta_0>0$, \ $\epsilon > 0$, \ $\gamma_0>0$\
, \ $0 \leq \tau_1 < \tau_2< 0.5 < \tau_3<1$, \\
$0<\eta_1<\eta_2\leq 0.5<\eta_3<1<\eta_4$, $\alpha=1$
\STATE Compute $g_0$
\FOR{$k=0,1,2,\ldots$}
        \IF{$\|g_k\|\leq\epsilon$}
        \RETURN
        \ENDIF
        \STATE Choose at most $m$ pairs $\{s_j, y_j\}$
	\STATE Compute $p^*$ using Algorithm \ref{alg:SOBS}
	\STATE 
	Compute step-size $\alpha$ with Wolfe line-search on $p^\ast.$ Set $p^\ast = \alpha p^\ast.$
        \STATE Compute the ratio $\rho_k = (f(w_k+p^*)-f(w_k))/\mathcal{Q}_k(p^*)$
                \STATE $w_{k+1}=w_k+p^*$
                \STATE Compute $g_{k+1}$, $s_k$, $y_k$, and $\gamma_k$
         \IF{$\rho_k < \tau_2$}
                \STATE $\delta_{k+1} = \min \left({\eta_1}\delta_k, {\eta_2}\|s_k\|_2 \right)$
        \ELSE
                \IF{$\rho_k \geq \tau_3$ \AND $\|s_k\|_2 \geq {\eta_3} \delta_k$\
}
                        \STATE $\delta_{k+1} = {\eta_4} \delta_k$
                \ELSE
                        \STATE{$\delta_{k+1}=\delta_k$}
                \ENDIF
        \ENDIF
\ENDFOR
\end{algorithmic}
\end{algorithm}

\begin{algorithm}[h]
\begin{algorithmic}[1]
\STATE Compute  the Cholesky factor $R$ of $\Psi^T\Psi$;
\STATE Compute the spectral decomposition $RMR^T = U\hat{\Lambda}U^T$
(with $\hat{\lambda}_1 \le \cdots \le \hat{\lambda}_ k$);
\STATE Let $\Lambda_1 = \hat{\Lambda} + \gamma I$;
\STATE Let $\lambda_{\min} = \min\{\lambda_1, \gamma\}$, and let $r$ be its algebraic multiplicity;
\STATE Define $g_{\parallel} \defined  (\Psi R^{-1}U)^Tg$;  
\IF{$\lambda_{\min} > 0$ \textnormal{ and } $\phi(0) \ge 0$ [the unconstrained minimizer is feasible]}
	\STATE 	$\sigma^* = 0$ and compute $p^* = -B_k^{-1}g_k$;
\ELSIF{$\lambda_{\min} \le 0$ \textnormal{ and } $\phi(-\lambda_{\min}) \ge 0$}
		\STATE $\sigma^* = -\lambda_{\min}$;
		\STATE 	Solve $(B_k + \sigma^* I)p^* = -g_k$;
		\IF{$\lambda_{\min} < 0$ [the hard case]}
			\STATE Compute $\alpha^*$ and $u_{\text{min}}^*$;
			\STATE $p^* \leftarrow p^* + \alpha^*u_{\text{min}}^*$;
		\ENDIF
\ELSE
		\STATE Use Newton's method to find $\sigma^*$, a root of $\phi$, in $(\max\{-\lambda_{\min}, 0 \}, \infty)$;
		\STATE Solve $(B_k + \sigma^* I)p^* = -g_k$;
\ENDIF
\caption{Orthonormal Basis SR1 Method}\label{alg:SOBS}
\end{algorithmic}
\end{algorithm}

\subsection{Stochastic extension} 
In this section, we describe how to improve the efficiency
of \LSRTR{} by incorporating approximate gradient calculations
derived from random sampling of the training data.  
The use of
  mini-batches can be motivated by considering (\ref{eqn-batch}), which
  suggests a (potentially significantly smaller) subset may be sufficient to obtain a meaningful
  descent direction for the true objective function.  Mini-batching
refers to the process whereby a subset of training data
is used to approximate the full gradient calculation each iteration.  That is, instead of using $g_k$ the gradient is approximated by
\begin{equation}\label{eqn-approxgrad}
	\tilde{g}_k \defined \frac{1}{|\Iset_k|} \sum_{i \in \Iset_k} \nabla f_i(w_k) \approx  \nabla f(w_k),
\end{equation}
where $\Iset_k \subseteq \{ 1, 2, \dots, n \}$.  
Obviously as $|\Iset_k|$ decreases the savings in computational cost
must be weighed against the resulting degradation in progress.
Remarkably first-order algorithms like \SGD{} function behave quite well
even if $\Iset_k$ consists of only a single observation at each iteration.
The reason is that the gradient error can be shown to cancel itself
out in the expected value sense.  However, for higher-order approaches such as quasi-Newton methods, the batch size typically
needs to be larger. Further, batch sizes need not be fixed--strategies for dynamically increasing batch size
have been studied in~\cite{ByChNoWu12,Me17,SmKiLe17}.
In our experience, we have found robustness in starting with an
arbitrarily small batch size and increasing the batch size whenever
progress towards the minimizer appears to stagnate.

For this work, we use overlapping training samples~\cite{BeNoTa16}, 
requiring that at each iteration
the mini-batch $\Iset_k$ is formed using a prescribed percentage of overlap
with the previous mini-batch.  That is, at the $k$th iteration, the
overlap $\Iset_k\cap \Iset_{k-1}$ is predetermined.  Using overlapping
mini-batches and (\ref{eqn-approxgrad}), the
quasi-Newton pairs $\{(s_{k-1}, y_{k-1})\}$ are computed as
$$
	s_{k-1} = w_k - w_{k-1} \quad \text{and} \quad y_{k-1} = \tilde{g}_{k} - \tilde{g}_{k-1}.
$$
As with SGD, there is inherent noise in the search direction 
due to using (\ref{eqn-approxgrad}) instead of the true gradient.
A common approach to mitigate the effects of this noise is to use the
principles of momentum, which is the exponential averaging of recent
steps.  Specifically, in our approach
we add the following momentum term at the end of each iteration:
\[
v_{k} = \mu v_{k-1} + (w_{k} - w_{k-1}).
\]
The most commonly-used value for the momentum parameter is $\mu=.9$
(see e.g., \cite{Sutskever2013}).
The momentum step $v_k$ is grafted into the 
trust-region solution $p^*$ from \eqref{eq:TRsubproblem} as follows:
\begin{subequations}
\begin{align}
	\label{eq:vk_momentum}
  v_k &\leftarrow \mu \min\left(1.0, \dfrac{\delta_k}{\|v_k\|} \right) v_k,  \\
  	\label{eq:pstar_momentum}
  p^* &\leftarrow \min\left(1.0, \dfrac{\delta_k}{\|
  p^*+v_k\|}\right)(p^*+v_k),
\end{align}
\end{subequations}
where $\delta_k$ denotes the current trust-region radius (see Algorithm~\ref{alg}).
Note that if $\mu = 0$, then the trust-region step $p^*$ would be left
unchanged by the above transformation. 
We call this approach Limited-Memory Stochastic \SR1{}
Trust-Region (\LSSRTR{}), and it differs from Alg.\ \ref{alg} (\LSRTR{})
in three specific places: Line 1, which uses the approximate gradient $\tilde{g}_0$
instead the exact initial gradient $g_0$;
Line 7, which incorporates the momentum step $v_k$ into the
trust-region subproblem solution $p^*$; and 
Line 11, which uses the approximate gradient 
$\tilde{g}_{k+1}$ instead the exact initial gradient $g_{k+1}$
and would compute $y_k$ using the approximate gradient, i.e., 
$y_k = \tilde{g}_{k+1} -\tilde{g}_k$.  \LSSRTR{} is outlined in Algorithm \ref{LSSR1}.  

\LSSRTR{} 
requires the use of two new hyper-parameters (the
momentum parameter and the mini-batch overlap parameter).
Unlike \SGD{}  where convergence is very sensitive
to the learning rate, we have found that convergence of the 
proposed method is not adversely affected by small changes in 
these hyper-parameters.
In fact, we have found that these parameters
are no more sensitive to tuning than the existing
quasi-Newton parameters such as memory size and the trust-region
expansion and contraction parameters (see $\eta_1$ and $\eta_2$ in Algorithm~\ref{alg}).

\begin{algorithm}[htp]
\caption{Limited-Memory Stochastic SR1 Trust-Region (L-SSR1-TR) Method} \label{LSSR1}
\begin{algorithmic}[1]
\REQUIRE $x_0\in \Re^n$, \ $\delta_0>0$, \ $\epsilon > 0$, \ $\gamma_0>0$\
, \ $0 \leq \tau_1 < \tau_2< 0.5 < \tau_3<1$, \\
$0<\eta_1<\eta_2\leq 0.5<\eta_3<1<\eta_4$, $\alpha=1$, $\mu = 0.9$
\STATE Compute initial batch $\Iset_0$ and $\hat{g}_0$
\FOR{$k=0,1,2,\ldots$}
        \IF{$\|\hat{g}_k\|\leq\epsilon$}
        \RETURN
        \ENDIF
        \STATE Choose at most $m$ pairs $\{s_j, y_j\}$
	\STATE Compute $p^*$ using Algorithm \ref{alg:SOBS}
	\STATE $v_k = \mu v_{k-1} + (w_k - w_{k-1})$
	\STATE $v_k = \mu \min( 1.0, \delta_k/\|v_k \|) v_k$
	\STATE $p^* = \min (1.0, \delta_k / \| p^* + v_k \|)(p^* + v_k)$
	\STATE Compute step-size $\alpha$ with Wolfe line-search on $p^\ast.$ Set $p^\ast = \alpha p^\ast.$ 
        \STATE Compute the ratio 
       $\rho_k = (\hat{f}(w_k+p^*)-\hat{f}(w_k))/\mathcal{Q}_k(p^*)$
                \STATE $w_{k+1}=w_k+p^*$
                \STATE Compute $\hat{g}_{k+1}$, $s_k$, $y_k$, and $\gamma_k$
        \IF{$\rho_k < \tau_2$}
                \STATE $\delta_{k+1} = \min \left({\eta_1}\delta_k, {\eta_2}\|s_k\|_2 \right)$
        \ELSE
                \IF{$\rho_k \geq \tau_3$ \AND $\|s_k\|_2 \geq {\eta_3} \delta_k$\
}
                        \STATE $\delta_{k+1} = {\eta_4} \delta_k$
                \ELSE
                        \STATE{$\delta_{k+1}=\delta_k$}
                \ENDIF
        \ENDIF
         \STATE Compute batch $\Iset_{k+1}$ in accordance with 
         Assumption \ref{as:stall1} 
\ENDFOR
\end{algorithmic}
\end{algorithm}

\subsubsection{Line-search analysis}

Here, we demonstrate that under some mild assumptions,
the line-search step in Algorithm \ref{LSSR1} 
is guaranteed to a step length that sufficiently decreases $\hat{f}(w)$.
We first state these assumptions.

\begin{assumption}\label{as:batch} Let the mini-batch set of observations 
    $\Iset_k$ be sampled randomly with $n_b=|\Iset_k|.$  Then there 
    exists a positive function $\gamma: \mathbb{R} \to \mathbb{R} $ 
    such that:
\begin{equation}\label{eqn-minibatch}
    \left \|
    \left \{ 
    	\dfrac{1}{|\Iset_k|}\sum_{i\in \Iset_k} \nabla f_i(w) 
    \right \}
    - \nabla f(w) \right \|_\infty\le \gamma(n_b)
\end{equation}
    where $\gamma(n_b) \to 0$ as $n_b \to n$. 
\end{assumption} 
\noindent 
This assumption suggests that as $|\Iset_k|$ increases, $\tilde{g}_k$  in 
\eqref{eqn-approxgrad} approaches $\nabla f(w)$.
\begin{assumption}\label{as:lnsrch}  
The line search in Algorithm \ref{LSSR1} is performed only 
on the sampled function $\hat{f}(w)$.
\end{assumption}
\noindent This assumption requires that the line search uses 
the same batch that was used to define the trust-region subproblem.
\begin{assumption} 
Algorithm \ref{LSSR1}
negates the search direction whenever
$\hat g(w)^T p^\ast > 0$, where $p^*$ is from \eqref{eq:pstar_momentum}.
\end{assumption}

\noindent Next, we make the following assumption to ensure
that we are making progress in decreasing the full
empirical risk $f(w)$.

\begin{assumption}\label{as:stall1} The objective $f$ is fully evaluated 
    every $J>1$ iterations 
    (say, at iterates $w_{J_0}, w_{J_1}, w_{J_2}, \dots$,
    where $0 \le J_0 < J$ and $J = J_1 - J_0 = J_2 - J_1 = \cdots$)
    and nowhere else in the algorithm.
    The batch size $n_b$ is monotonically
    increased whenever 
    \[	 f(w_{J_{\ell}}) >
    f(w_{J_{\ell-1}}) - \tau\]
    for some $\tau > 0.$
\end{assumption} 
\noindent This assumption states that if progress is not made in 
decreasing $f$, the batch size is increased to reduce the noise associated
with using a subsampled surrogate function $\hat{f}$.

\bigskip

Given Assumptions 1 through 4, we now present convergence results
for  \LSRTR{}.  The theorem below asserts the trust-region radius
update will always succeed.
\begin{theorem}
	At iteration $k$, given the batch $\Iset_k$,
	the line-search step in Algorithm \ref{LSSR1} 
	can never fail.
    That is, there
    exists $\alpha > 0$ such that the strong-Wolfe conditions hold:
    \begin{enumerate}
        \item $\hat{f}(w_k + \alpha p^\ast) \le \hat{f}(w_k) + c_1\alpha \nabla\hat{f}(w_k)$
        \item $|\nabla\hat{f}(w_k+\alpha p^\ast)^T p^\ast| \le  c_2 | \nabla\hat{f}(w_k)^T p^\ast|.$
    \end{enumerate}
\end{theorem}
\begin{proof}
    Because each $f_i(w)$ is smooth, the function $\hat{f}(w)$ 
        is likewise smooth.  Thus because the search direction
        is a descent direction for $\hat{f}(w)$, the result follows.
    Because of Assumption \ref{as:lnsrch} and smoothness assumptions on
    elements $f_i(w)$, classical line-search proofs hold so long
    as the batch $\Iset_k$ is held constant and not resampled during this 
    stage.
\end{proof}  
\begin{theorem}
    If the momentum parameter $\mu \to 0$,
    then either  
    \begin{equation}\label{eq:nbn}
    \liminf_{k\to\infty}\|\nabla f(w_k)\| = 0
    \quad \text{or} \quad
    \liminf_{k\to\infty} f(w_k) = -\infty.
    \end{equation}
\end{theorem}
\begin{proof}
    For simplicity of notation, we
    will define $\hat w_i = w_{J_i}.$ 
    By {Assumption}~\ref{as:stall1}, the 
    objective function must monotonically reduce
    over the subsequence $\{ \hat{w}_i \}$
    or $n_b\to n.$  Suppose the objective function is decreased
    $\iota_k$ times over the subsequence $\{ \hat{w}_i \}_{i=0}^{k}$.  
    Then
    \[  
f(\hat w_k) = f(\hat w_0) + \sum_{i=1}^{\iota_k} (f(\hat w_i) - f(\hat w_{i-1}))
\le f(\hat w_0) - {\iota_k}\tau.
  \]  
  Assuming $n_b \nrightarrow n$, then as $k \rightarrow \infty$,
  $\iota_k \rightarrow \infty$, and \eqref{eq:nbn} holds.
  If $n_b \rightarrow n$, we reduce to a classic
    line-search approach whose convergence is assured via the 
    trust-region algorithm that makes sufficient progress at each iteration
    (see e.g., \cite{NocY98}).
 \end{proof}

\subsection{Initial matrix $B_0$} 

In this section we borrow terminology defined in Section~\ref{section:SolveLSR1}.
For simplicity in this section we will assume that $B_0=\gamma I$
and analyze the impact of $\gamma$ on various scenarios.  We will
show in this section that the choice of $\gamma$ plays a critical
role in a trust-region approach.  We start by proving a brief lemma
summarizing how directions of negative curvature present in $B_k$ affects the
trust-region solution, a variation of which may also be found in~\cite{Yektamaram17}.
We will denote the smallest and largest eigenvalues of $B_k$ 
by $\lambda_{\min}(B_k)$ and $\lambda_{\max}(B_k)$, respectively.
\begin{lemma}\label{lemma:Bk}
    If $B_k \not \succeq 0,$ as the trust-region radius
    $\delta_k$ increases, the trust-region solution, $p^\ast,$ 
    asymptotically becomes
    parallel to the eigenspace corresponding to $\lambda_{\min}(B_k)$.
    That is,
     \[
    \lim_{\delta_k \to \infty} 
    \dfrac{|u_{\rm min}^T p^\ast|}{\|u_{\rm min}\| \| p^\ast\|}
    = 1.
    \]
    where $u_{\rm min}$ is an eigenvector corresponding to
    $\lambda_{\min}(B_k)$.
\end{lemma}
\begin{proof}
    Without loss of generality, we assume that
    the $\lambda_{\min}(B_k)$ has multiplicity one
    for ease of presentation. (For how to handle the general case,
    the notation in \cite{BruEM17} can be used.)      
    Let $\lambda_{\min}(B_k) = 
    \lambda_1 < \lambda_2 \le \dots \le \lambda_n$, and let
    $\pi_i$ be the $i^{\text{th}}$ column of
    $\Pi$ in the eigendecomposition of $B_k$ in \eqref{eqn-Beig}.
    Using this notation, $u_{\min} = \pi_1$.
    Then we can define 
    \begin{equation}\label{eq:pnorm}
    	\| p(\sigma)\|^2 
	= 
	\| (B_k+\sigma I)^{-1}g_k\|^2 
	=
	\| \Pi(\Lambda + \sigma I)^{-1}\Pi^T\!g_k \|^2 
	=
	 \sum_{i=1}^n \frac{(\pi_i^Tg_k)^2}{(\lambda_i + \sigma)^2},
    \end{equation}
    provided $\sigma \ne -\lambda_i$ for $1 \le i \le n$.
    To prove the lemma, we consider two cases: 
    (i) $\pi_1^Tg_k \ne 0$ and (ii) $\pi_1^Tg_k = 0$.
    \\     \\
    \textsl{Case (i):} If $\pi_1^Tg_k \ne 0$, then rearranging \eqref{eq:pnorm}
    yields 
    $$
    	\frac{1}{(\lambda_{\min}(B_k)+\sigma)^2} 
	=
	\frac{1}{(u_{\min}^Tg_k)^2} 
	\bigg ( \| p(\sigma) \|^2 - \sum_{i=2}^n \frac{(\pi_i^Tg_k)^2}{(\lambda_i+\sigma)^2} \bigg ),
    $$
    since $\lambda_{\min}(B_k) = \lambda_1$ and $u_{\min} = \pi_1$.
    Moreover, as $\sigma \rightarrow -\lambda_{\min}(B_k)^+$, then
    $\| p(\sigma) \| \rightarrow \infty$ and $\lim_{\sigma
    \rightarrow -\lambda_{\min}(B_k)^+} \phi(\sigma) = -1/\delta$ (see \eqref{eq:secular}).
    Since $\phi(\sigma^*) = 0$ and $\phi(\sigma)$ is continuous on the interval
    $(-\lambda_{\min}(B_k), \infty)$, the optimal Lagrange multiplier $\sigma^*$
    satisfies $\sigma^* > -\lambda_{\min}(B_k)$ (see Fig.\
    \ref{fig:phi}(a)).
    Thus, the solution $p^* = -(B_k +\sigma^*I)^{-1}g_k$ satisfies
    $$
    u_{\min}^Tp^* 
    = -u_{\min}^T(B_k + \sigma^*I)^{-1}g_k
    = -\frac{u_{\min}^Tg_k}{\lambda_{\min} + \sigma^*}.
    $$
    At the optimal Lagrange multiplier $\sigma^*$,
    the trust-region subproblem solution $p(\sigma^*)$
    lies on the boundary, i.e., $\| p(\sigma^*) \| = \delta_k$ and 
    since $\| u_{\min} \| = \| \pi_1\| = 1$, we have
    \begin{eqnarray*}
        		\lim_{\delta_k \to \infty} 
    		\dfrac{|u_{\rm min}^T p^\ast|}{\|u_{\rm min}\| \| p^\ast\|}
    		&=&
           	\lim_{\delta_k \to \infty} 
    		\frac{|u_{\min}^Tg_k|}{(\lambda_{\min} + \sigma^*)\delta_k} 
		  \\
    		&=&
           	\lim_{\delta_k \to \infty} 
		\bigg ( \delta_k^2 
		- 
		\sum_{i=2}^n 
		\frac{(\pi_i^Tg_k)^2}{(\lambda_i+\sigma^*)^2} \bigg )^{\tfrac{1}{2}} 
		\cdot \frac{1}{\delta_k} \\
		&=& 1.
    \end{eqnarray*}
    \textsl{Case (ii):}
    Suppose $\pi_1^Tg_k = 0$.  For any $\sigma^* \ge -\lambda_{\min}$, 
    the vector $\hat{p}^*$ given by
    $$
    	\hat{p}^*
	= -(B_k+\sigma^*)^{\dagger}g_k
	=
	-\sum_{i=2}^n
	\frac{\pi_i^Tg_k}{\lambda_i+\sigma^*} \pi_i,
    $$    
    satisfies the first optimality condition $(B_k+\sigma^*)\hat{p}^* = -g_k$.  
    Now the length of $\hat{p}^*$ is bounded since 
    $\sigma^* \ge -\lambda_1 > -\lambda_i$
    for all $i \ge 2$.  Thus, for sufficiently large $\delta_k$, 
    $\| \hat{p}^* \| < \delta_k$, and the
     trust-region subproblem solution is given by
    $$
    p^\ast = \hat p^\ast + \alpha u_{\rm min},    
    $$  
    where $\alpha$ is chosen such that $\| p^*\| = \delta_k$ (see Sec.\ 
    \ref{section:SolveLSR1}). 
    (Note that this is precisely the hard case (see Fig.\ \ref{fig:phi}(b).)
    Since $u_{\min}^T\hat{p}^* = 0$ (see \cite{BruEM17}),
    $$
    	    \lim_{\delta_k \to \infty} 
    \dfrac{|u_{\rm min}^T p^\ast|}{\|u_{\rm min}\| \| p^\ast\|}
    =
    	    \lim_{\delta_k \to \infty} 
    \dfrac{|\alpha|}{ \| p^\ast\|}
    =
    	    \lim_{\delta_k \to \infty} 
    \dfrac{\sqrt{\delta_k^2 - \|\hat{p}^*\|^2}}{\delta_k}
    =
    1,
    $$
    which completes the proof.
    \begin{figure}[t]
    \label{fig:phi}
    \centering
    \includegraphics[height=4.8cm]{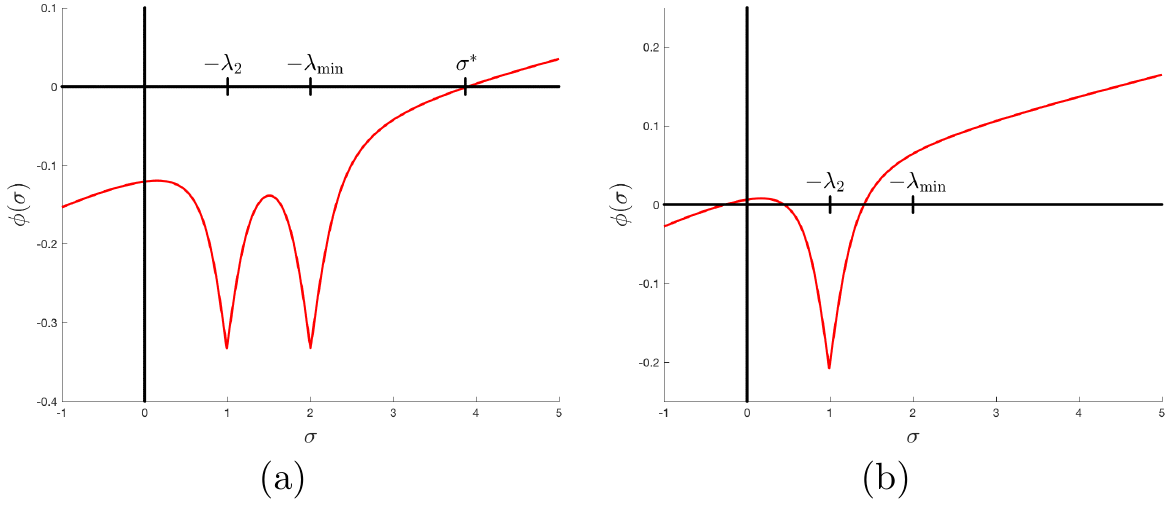} 
    \caption{
    Graph of $\phi(\sigma)$ when $B_k \not \succeq 0$. 
    (a) When $\pi_1^Tg \ne 0$, the optimal Lagrange multiplier satisfies
    $\sigma^* > -\lambda_{\min}(B_k)$. 
    (b) As $\delta_k \rightarrow \infty$, $\phi(\sigma)$ behaves 
    as in the hard case: If $\pi_1^Tg_k = 0$ and 
    $\lim_{\sigma \rightarrow -\lambda_{\min}(B_k)^+} \phi(\sigma) > 0$,
    then $\hat{p}^* = -(B_k+\sigma^*)^{\dagger}g_k$ has length
    $\| \hat{p}^* \| < \delta_k$, where $\sigma^* = -\lambda_{\min}$.
    In this case, $p^* = \hat{p}^* + \alpha u_{\min}$, where
    $u_{\min}$ is an eigenvector of $B_k$ corresponding to  
    $\lambda_{\min}(B_k)$ and $\alpha$ is chosen so that $\| p^* \| = \delta_k$.
    }
    \end{figure}
\end{proof}
Lemma \ref{lemma:Bk} shows the importance for $B_k$ to capture
curvature information correctly since the trust-region
subproblem solution, $p^*$, becomes more parallel to the
eigenvector corresponding to the most negative eigenvalue of $B_k$.
We next prove conditions that highlight how the choice of $\gamma$
affects $B_k$.
\begin{lemma}\label{lemma:geneig}
    Suppose $B_0=\gamma I$ and that $\hat \lambda$ denotes the smallest eigenvalue of the 
    generalized eigenvalue problem 
    \[
    (D_{k} \!+\! L_{k} \!+\! L_{k}^{\transpose}) u = \hat \lambda \! S_{k}^{\transpose}\!S_{k} u.
    \]
    Further assume that $\Psi_k$ and $S_k$ are 
    {full rank}.
    Then if $\hat \lambda > 0,$ we have the following properties:
    \begin{enumerate}
        \item $B_k$ is positive definite if $0 < \gamma < \hat \lambda.$ \label{enum:eig1}
        \item As $\gamma \to \hat \lambda$ from below, 
        $\lambda_{\max}(B_k) \to \infty$
        and ${\rm cond}(B_k) \to \infty.$
        \item As $\gamma \to \hat \lambda$ from above, 
        $\lambda_{\min}(B_k) \to -\infty$.  
    \end{enumerate}
\end{lemma}
\begin{proof}
 Recall from \eqref{eqn-compactSR1} and \eqref{eq:PsiM} that
$B_{k} \ = \ \gamma I +  \Psi_{k} M_{k} \Psi_{k}^T,$    
where  $\Psi_k = Y_{k} - \gamma S_{k}$ and
    $ M_{k}^{-1}$ 
    $=$ $D_{k} \!+\! L_{k} \!+\! L_{k}^{\transpose} \!-\! \gamma S_{k}^{\transpose}\!S_{k}.$
    Note 
    if $0 < \gamma < \hat \lambda$, then 
    {$M_{k}^{-1} \succ 0$, and consequently, $M_{k} \succ 0$.}
    Hence
    $B_k \succ 0.$  By assumption we have
    \[
    M_{k}^{-1}u =(\hat \lambda - \gamma) S_{k}^T S_{k} u. 
    \]
    Forming the QR factorization of $S_k=Q_kR_k$
    and letting $z=R_ku$, we have that
    $S_k^T S_k= R_k^T R_k$ and
    $R_k^{-T} M_k^{-1} R_k^{-1} z = (\hat \lambda - \gamma) z.$
    Consequently, 
    \[    
    R_k M_k R_k^T z = \dfrac{1}{(\hat \lambda - \gamma)}z.
    \]
    {Let $\hat{z}$ be} the min-two norm solution to $\Psi_k^T \hat z = z$.
    Then we have that
    \[
    \hat z^T B_{k} \hat z = \gamma \|\hat z\|^2 + \dfrac{1}{\hat\lambda - \gamma}\|z\|.^2 
    \]
    The results then follow since $\hat z$ and $z$ are constant and
    nonzero as $\gamma$ changes.
\end{proof}
Lemma \ref{lemma:geneig} shows that not choosing $\gamma$
judiciously in relation to $\hat{\lambda}$ can have
deleterious effects.  In particular, if $\gamma$ is too close to 
$\hat{\gamma}$ from below, then $B_k$ becomes ill-conditioned.
If $\gamma$ is too close to 
$\hat{\gamma}$ from above, then the smallest eigenvalue
of $B_k$ becomes negatively large arbitrarily.

Next, we analyze the relationship between the choice 
of $\gamma$ and the conditions under which we can expect
$\hat{\lambda} > 0$.  First, we note that
the predicted reduction $\mathcal{Q}_k(p^*)$ $=$ $\nabla f(w_k)^T p^* + \tfrac{1}{2} {p^*}^T B_kp^*$
is always less for a descent direction $p^*$ if $p^*$ is also a
direction of negative curvature.  Moreover, from Lemma \ref{lemma:Bk},
we see that $p^*$ tends to be parallel to the eigenvector of
$B_k$ corresponding to its smallest eigenvalue.  
However, it is desirable to avoid the situation where
$p^*$ is a \emph{false} direction of negative curvature, meaning
${p^*}^T \nabla^2 f(w_k) p^* > 0$ while 
${p^*}^T B_k p^* < 0.$ The following lemma shows that in the limit,
we can select $\gamma$ so that $0 < \gamma < \hat{\lambda}$,
i.e., in the limit, $\hat{\lambda}> 0$ unless the true underlying
Hessian is either indefinite or singular.

\begin{lemma}
    Suppose that $f$ is twice-continuously differentiable,
    that the matrix $S_k$ remains full-rank,
    and that $w_k \to w^*$,  
    where 
    $
    \nabla^2 f(w^*) \succ 0.
    $
    Then $\hat \lambda$ corresponding to $B_{k}$ is 
    positive in the limit.
\end{lemma}
\begin{proof}
    We observe that each $(s_{j}, y_{j})$ pair satisfy
    $y_{j} = \nabla f(w_{j} + s_{j}) - \nabla f(x_{j}).$
    Using Taylor expansion, we have that
    \begin{eqnarray*}
        \nabla f(w^* - w^* + w_{j} + s_{j}) &=& \nabla f(w^*)
        + \nabla^2 f(w^*) (w_{j} - w^* + s_{j})  + 
       { t_{j+1}} \\
        \nabla f(w^* - w^* + w_{j} )&=& \nabla f(w^*)
        + \nabla^2 f(w^*) (w_{j} - w^*)  + 
       { t_j},
    \end{eqnarray*}
    {where the components of $t_{j+1}$ and $t_j$ are 
    ${\rm O}(\|w_{j} -  w^* + s_{j}\|^2)$ and
    ${\rm O}(\|w_{j} - w^*\|^2)$, respectively.
    Combining these two equations yields}
    \begin{equation}\label{eq:yj}
    y_j =  \nabla^2 f(w^*) s_j +  
    {(t_{j+1}-t_j)}.
    \end{equation}
    We must prove that there exists a $K  >0$ and $\beta >0$ such that for 
    all $k > K$,
    \[
    \hat \lambda = \min_{v} \dfrac{ v^T (L_k + D_k + L_k^T) v}{v^T S_k^T S_k v}
    > \beta .
    \]
    For simplicity let us define $A_k=L_k + D_k + L_k^T$ such that
    \[
    (A_k)_{i+1,j+1} = \left \{ \begin{array}{ll}
    s_i^T y_j &\text{ for }  i \ge j \\
    s_j^T y_i & \text{ otherwise.} 
    \end{array}\right.
    \]  
    Note from \eqref{eq:yj} we have that
    \[
    s_i^T y_j=  s_i^T\nabla^2 f(w^*) s_j + 
    {s_i^T(t_{j+1}-t_j)},
    \]
    and thus 
    \begin{eqnarray*}
        v^T (L_k + D_k + L_k^T) v   &=& 
        v^T S_k^T \nabla^2 f(w^*) S_k v  + 
        \\
        &&
	 2\sum_{j=k-r}^{k-1} \sum_{i=j}^{k-1} s_i^T(t_{j+1}-t_j)v_{i+1}v_{j+1}
	+ \sum_{j=k-r}^{k-1} s_j^T(t_{j+1}-t_j) v_{j+1}^2.
    \end{eqnarray*}
    {Since $s_j = w_{j+1}-w_j$, as $w_j$ converges to $w^*$, 
    both $t_{j+1}$ and $t_j$ tend to 0.}
    Thus
    \[
    \lim_{k \to \infty} \dfrac{ v^T (L_k + D_k + L_k^T) v}{v^T S_k^T S_k v}
    =
    \lim_{k \to \infty}  \dfrac{ v^T S_k^T \nabla^2 f(w^*) S_k v}{v^T S_k^T S_k v}
    \ge \lambda_{\rm min} (\nabla^2 f(w^*)) > 0
    \]
    by assumption.
\end{proof}
In the next lemma, we show that selecting $\gamma > \hat \lambda$
can result in a false curvature prediction. To simplify the proof
we show that the result holds for a quadratic function.  A
more general proof simply uses Taylor expansions and
asymptotic limit properties.
\begin{lemma}
    Suppose we apply Algorithm \ref{alg} to a quadratic objective function
    {$f(w) = c^Tw + \tfrac{1}{2} w^THw$, where $c \in \Re^n$ and
    $H \in \Re^{n \times n}$ are both constant.}
    Then if $\gamma = \tau \hat \lambda$ with $0 < \tau < 1$
    then $B_{k}$ can be indefinite only if the true Hessian
    is indefinite in {the range of $S_k$}, 
    that is,
    \[
    S_k^T \nabla^2 f(w) S_k \not \succeq 0.
    \] 
    Conversely, if $\tau > 1$ then $B_{k}$ may have
    arbitrarily large negative eigenvalues even if the objective
    is convex.  Furthermore, for any trust-region radius
        $\delta_k>0,$
       \[
       	\lim_{\tau \rightarrow 1^{+}} \mathcal{Q}_k(p^*) = - \infty.
       \]
    Thus the model's 
     quality measured by the ratio of 
     actual reduction versus predicted reduction
\[  
	\rho_k = \dfrac{ f(w_k + p^*) - f(w_k)}{\mathcal{Q}_k(p^*). }  
\]
may be arbitrarily poor for any $\delta_k$ sufficiently large.
\end{lemma}
\begin{proof}
	Note that for a quadratic function $f(w)$, 
	$$
	y_k = \nabla f (w_{k+1}) - \nabla f(w_k) = Hw_{k+1} - Hw_k = Hs_k,
	$$
	and therefore, $Y_k = HS_k$.
    This implies that $S_k^TY_k = S_k^THS_k$, and  therefore,
    $L_k + D_k + L_k^T = S_k^THS_k$.
    Then from \eqref{eqn-compactSR1} and \eqref{eq:PsiM}, we have
    \[
    B_{k} = \gamma I + (H-\gamma I)S_k(S_k^T H S_k - \gamma S_k^T S_k)^{-1} 
    S_k^T (H - \gamma I).
    \]
    If $S_k^T H S_k \succ 0$, then 
    $L_k + D_k + L_k^T \succ 0$ and
    $\hat \lambda > 0$.  
    Thus, if $\gamma = \tau \hat{\lambda}$ with $0 < \tau < 1$, then
    $B_{k}$ is positive definite since 
    $(S_k^T H S_k - \gamma S_k^T S_k)$ is positive definite 
    because $0 < \gamma < \hat \lambda$. \\ \\
    Conversely, {if $\tau > 1$, from  the smallest eigenvalue of
    $(S_k^T H S_k - \gamma S_k^T S_k)$ is negative.
    Then as $\tau \to 1^+$, 
    $\lambda_{\min}(S_k^T H S_k - \gamma S_k^T S_k)$ approaches $0^-,$ 
    implying $\lambda_{\min}(B_{k})$ approaches $-\infty.$}
    Let $\pi_{\min}$ denote a vector in the eigenspace corresponding
to $\lambda_{\rm min}(B_k),$ scaled so that $\|\pi_{\min}\|=\delta_k.$ 
Then 
\begin{equation}\label{eq:qpstar}
	\mathcal{Q}_k(p^*) \le 
	\mathcal{Q}_k(\pi_{\min}) 
	= c^T \pi_{\min}+ \dfrac{1}{2}\pi_{\min}^T B_k\pi_{\min}
	\le \|c\|\delta_k + \dfrac{1}{2}{\lambda_{\rm min}(B_k)}\delta_k^2.
\end{equation}
Thus $\lim_{\tau \rightarrow 1^{+}} \mathcal{Q}_k(p^*) = - \infty.$
Moreover, for $\tau$ sufficiently close to $1$ from above,
$B_k$ is indefinite, i.e., $\lambda_{\min}(B_k) < 0$, and therefore
$\lim_{\delta_k \to \infty} \mathcal{Q}_k(p^*
    ) = -\infty$ in \eqref{eq:qpstar}.
    In contrast, 
    if the quadratic objective function is convex, then we must have
    \[
    \lim_{\delta_k \to \infty} f(w_k+p^*) - f(w_k)
    = 
     \lim_{\delta_k \to \infty} c^T p^* + \frac{1}{2}(p^*)^THp^*
     = \infty,
    \]
    meaning that for sufficiently large $\delta_k$,
    the model function $\mathcal{Q}_k$ 
    poorly predicts the actual reduction in $f$.
\end{proof}
When combined with Lemma \ref{lemma:geneig}, 
the following lemma suggests selecting a $\gamma$ 
near but strictly less than $\hat \lambda$ to avoid asymptotically
poor conditioning while improving the negative curvature approximation
properties of $B_k.$   
Note that $\hat \lambda$ is cheaply determined due
to the small column dimension of $S_k.$
\begin{lemma}
    Suppose we apply Algorithm \ref{alg} to a quadratic objective function
    $f(w) = c^Tw + \tfrac{1}{2} w^THw$, where $c \in \Re^n$ and
    $H \in \Re^{n \times n}$ are constant.
    Let $\hat \lambda$ denote the smallest eigenvalue of the 
    generalized eigenvalue problem 
    \[
    (D_{k} \!+\! L_{k} \!+\! L_{k}^{\transpose}) u = \hat \lambda \! S_{k}^{\transpose}\!S_{k} u.
    \]
    Then for all $\gamma < \hat{\lambda}$, 
    the smallest eigenvalue of $B_k$ is bounded above by 
    the smallest eigenvalue of $\nabla^2 f(w) = H$ in the span of $S_k$, i.e.,
    $$
    	\lambda_{\min}(B_k) \le 
	\min_{S_k\!v \ne 0} \frac{v^TS_k^THS_kv}{v^TS_k^TS_kv}.
    $$
\end{lemma}
\begin{proof}
If $\gamma < \hat{\lambda}$, then the matrix 
$D_k + L_k + L_k^T - \gamma S_k^TS_k$ is positive definite
from the definition of $\hat{\lambda}$.
Therefore, from \eqref{eqn-compactSR1} and  \eqref{eqn-Beig} ,
the eigenvalues in the matrix  $\hat{\Lambda}$
for the low-rank update to $B_0$  are positive.  Consequently, 
$
	\lambda_{\min}(B_k) = \gamma.
$
From the proof of Lemma 2.6, 	$L_k + D_k + L_k^T = S_k^THS_k$.
Therefore,
$$
	\hat{\lambda} = \min_{S_k\!v \ne 0}
	\frac{v^TS_k^THS_kv}{v^TS_k^TS_kv}.
$$
The result follows from the assumption that $\gamma < \hat{\lambda}$.
 \end{proof}

It is further worthwhile
to note that these observations were motivated by investigating
why the algorithm failed on some test cases but not others. Once
these safe-guards were put in place, the robustness of the algorithm
went from inferior to L-BFGS to superior. That is, if the 
reader has attempted to use L-SR1 in the past and found sometimes
it works great, and other times it fails, we suggest that it is
likely the case that failures were induced by inadvertently
permitting the case $0 < \hat \lambda < \gamma$
to occur.

\section{Numerical results}
\label{sec:method}
In this section, we present two sets of numerical results comparing the
performance of several methods, including the proposed Limited-Memory
\SR1{} Trust-Region (\LSRTR{}) and Limited-Memory Stochastic \SR1{}
Trust-Region (\LSSRTR{}) methods, on two databases.
In our experiments, we use a fully-connected network model (see Figure\ \ref{fig:network}).
The training inputs $x_i$ are $28 \times 28$ images, 
which are represented as vectors in $\Re^{28^2} = \Re^{784}$.  
At each layer, an affine transformation $W_{\ell}a_{\ell-1}+b_{\ell}$ is applied to the input vector
$a_{\ell-1}$, where $W_{\ell}$ is a matrix of weights and $b_{\ell}$ is a bias vector.
Before passing onto the next layer, an activation function $\theta$, defined to be the logistic function
$$
	\theta \left ((W_{\ell} a_{\ell-1} + b_{\ell}) \right )_j  = \frac{1}{1 + e^{-(W_{\ell} a_{\ell-1}+b_{\ell})_j}},
$$
is applied.  At the final layer, $L$, we apply a softmax function, given by
$$
	(\mathcal{S}(\theta(W_{L}a_{L-1} + b_{L})))_j
	=
	\frac{e^{(\theta(W_{L}a_{L-1} + b_{L}))_j }}
	{\sum_{k=1}^K e^{(\theta(W_{L}a_{L-1} + b_{L}))_k}},
$$
so that the output vector $p(w, x_i)$ corresponds to probabilities with
$
	\sum_{k=1}^K ( p(w, x_i) )_k = 1.
$
Here, $w = (W_1, b_1, W_2, b_2, \dots, W_L, b_L)$.
The softmax function is 
paired with cross-entropy for the final 
output layer to form the resulting loss function element 
$f_i$ in \eqref{eqn-min}:
$$
	f_i(w) = -\sum_{k = 1}^K  (y_i)_k \log (p(w,x_i))_k,
$$
where $K$ 
is the dimension of the output layer.
For further details, see \cite[Chap.\ 11]{Friedman2001}.
Finally, for \LSSRTR{}, we used 33\% for the overlap and used a minibatch size of 100, increasing the batch size by a factor of 1.5 when progress ceased relative to the true loss.

\begin{figure}[tb]
\centering
\includegraphics[width=12cm]{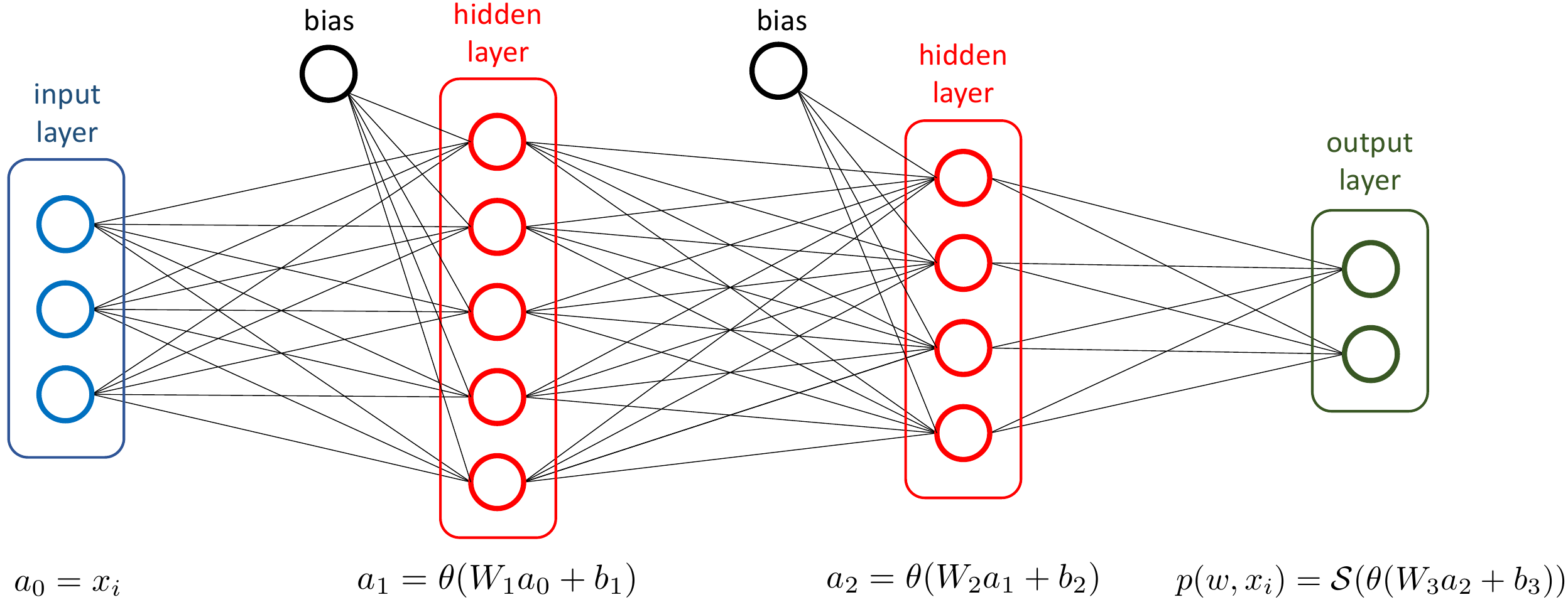}
\caption{
Illustration of a fully-connected network model.
Here, the input vector is $x_i \in \Re^d$ and the final output vector is
$p(w, x_i) \in \Re^K$, with $d = 3$ and $K = 2$.  
The weight matrices are $W_1 \in \Re^{5 \times 3}$,
$W_2 \in \Re^{4 \times 5}$, and $W_3 \in \Re^{2 \times 4}$.
The bias vectors are $b_1 \in \Re^5$, $b_2 \in \Re^4$, and $b_3 \in \Re^2$. 
The activation function $\theta( \ \cdot \ )$
is applied before passing the output to the next layer, except at the input layer.
The output vector $a_{\ell}$ of hidden layer $\ell$ is then used as the input in the next layer, $\ell+1$.
In the final (output) layer, the softmax function $\mathcal{S}( \ \cdot \ )$ is applied so that
the output vector $p(w,x_i)$ corresponds to probabilities with
$\sum_{k=1}^K (p(w,x_i))_k = 1$.
Here, $w = (W_1,b_1,W_2,b_2,W_3,b_3)$.  When vectorized, $w \in \Re^{54}.$
}
\label{fig:network}
\end{figure}

Two errors are used to train a network: training error and test
error.  The training error is used to define the optimization problem
(\ref{eqn-min}).  Most approaches that use \emph{training data} tend to
find models that \emph{overfit} the data, i.e., the models find
relationships specific to the training data that are not true in general.
In other words, overfitting prevents machine learning algorithms from
correctly generalizing.  To help prevent overfitting, an
independent data set, called the \emph{test set} is used to validate the
accuracy of the model to gage its usefulness in making future predictions.  Training
errors and test errors are computed using the loss function $f(w)$ in
(\ref{eqn-min}).
For machine learning, it is important to make sure the trained model yields
as small test
error as possible.  
The solution of (\ref{eqn-min}) is taken to be the $w$ that
minimizes the test error even though we are directly minimizing the
training error, which is our best measure for estimating 
the expected value of the loss function for unknown data.
Generally speaking, with neural network 
models it is possible to drive the training error to zero for sufficiently large networks; however, the resulting models tend to be overfitted and have
less predictive value.

\bigskip

\noindent \textbf{Experiment I.} 
For the first set of experiments, we compared the training and test 
errors of three methods: (i) a Hessian-free utilizing the Generalized Gauss-Newton 
method described in~\cite{maGN},
(ii) an \LBFGS{} method based on~\cite{LBFGS}, and (iii) the
proposed \LSRTR{} method
(see Figure \ref{fig:MNIST1}).
We do not include existing \SGD{} methods 
because they are already finely tuned for the \MNIST{} data set 
and the computational time involved in the hyper-parameter tuning 
cannot easily be accounted for in a fair comparison. For both \L-BFGS{}
and \LSRTR{} methods, a Wolfe line search was used.
We tested the three methods on 
two data sets with full training and testing observations.
The first set (Experiment IA) uses the full Mixed National
Institute of Standards and Technology (\MNIST) database, which is a
large collection of handwritten digits that is commonly used for
training various image processing
systems~\cite{Kussulernst2004,Lecun1998}.  It contains 60,000 training
images and 10,000 testing images.  The goal is to train the neural
network in order to classify the handwritten digits 0 through 9
with minimal error.
The second set (Experiment IB)  uses 
the Extended \MNIST{} (\EMNIST) database, which is an extension
of the \MNIST{} database to handwritten letters \cite{EMNIST}. 
We compared the
performance of the three methods on different network configurations
with varying numbers of layers and neurons, 
which are denoted by the sequence of numbers above each graph in Figures  \ref{fig:MNIST1}
and \ref{fig:EMNIST1}. For example, 
the sequence ``784-350-250-150-10" in Figure\ \ref{fig:MNIST1}(a) refers to the following:
the number of inputs is $784 = 28^2$, which corresponds to the pixel value of the input images, 
which are $28 \times 28$ in size;
the number of layers is $3$ with 350 neurons in the first layer,
250 in the second, and 150 in the third; and the number
of outputs is 10 for the 10 different classes that correspond to the digits from 0 to 9.
All tests were performed in {\small MATLAB}
(R2016b) on a 64-bit 2.67Ghz Intel\circledR \;Xeon \circledR \;CPU E7-8837  
machine with 4 processors and 256 GB RAM.
These experiments  were designed to test the hypothesis that one of
the primary reasons why Hessian-free methods outperform BFGS variants in
deep learning optimization problems is that they better approximate and
exploit negative curvature.

The results on the four different network configurations 
are given in
Figure \ref{fig:MNIST1}.
In Figure \ref{fig:MNIST1}, loss versus ``iterations'' and ``time'' are plotted.
Generally speaking, the Hessian-free method
outperforms both \LBFGS{} and 
{\small L-SR1-TR} in terms of achieving the
smallest test loss (and training loss) in the fewest iterations, with 
{\small L-SR1-TR}
outperforming \LBFGS{}.  
However, the cost per iteration for Hessian-free is significantly
higher since Hessian-free uses matrix multiplies whereas the quasi-Newton
methods use a (much cheaper) single
gradient evaluation.  Thus, in terms of wall-time, {\small L-SR1-TR}
 is the fastest method, obtaining 
the best solution in the least amount of time given a one-hour window to
solve the given network.  

\bigskip

\noindent 
\textbf{Experiment II.} 
The second set of experiments compares the two proposed \LSRTR{} and
the stochastic mini-batch version of \LSRTR{} (\LSSRTR) methods
on the same network configurations as in the first set of
experiments (see Figures \ref{fig:MNIST2} and  \ref{fig:EMNIST2}).  
While the \LSRTR{} method
achieves lower test and training losses than \LSSRTR{} per iteration
(see Figures \ref{fig:MNIST2}(a,c,e,g) and  \ref{fig:EMNIST2}(a,c,e)),
\LSSRTR{} is the fastest method in terms of wall-time 
(see Figures \ref{fig:MNIST2}(b,d,f,h) and \ref{fig:EMNIST2}(b,d,f)) 
because the computational cost per iteration for \LSSRTR{} is 
significantly cheaper.

We note that the results in Figures \ref{fig:MNIST1}-\ref{fig:EMNIST2}
are representative of other experiments.

\section{Conclusions}

\label{sec:conclusions}

In this paper, we presented an alternative approach for solving machine learning problems
that is based on the \LSR1{} update that allows for indefinite Hessian approximations.
This approach is particularly suitable for non-convex problems where exploiting
directions of negative curvature is crucial.  Numerical experiments suggest
that the proposed approaches (the limited-memory \SR1 trust-region and the limited-memory
stochastic \SR1 trust-region methods) can outperform the more commonly used quasi-Newton approach
(\LBFGS{}) both in terms of computational efficiency and test and training loss.

\section*{Acknowledgments}
We would like to thank Wenwen Zhou and Alireza Yektamaram for stimulating conversations
concerning stochastic quasi-Newton methods with SR1.  Further Alireza was instrumental in setting up the stochastic framework used to generate numerical results.   
J.\ Erway's research work was funded by NSF Grants CMMI-1334042 and IIS-1741264.
R.\ Marcia's research work was funded by NSF Grants CMMI-1333326 and IIS-1741490.


\bibliographystyle{abbrv}
\bibliography{myrefs,jerway,jgriffin,rmarcia}

\begin{figure}[htbp!]
\centering
\includegraphics[width=14cm]{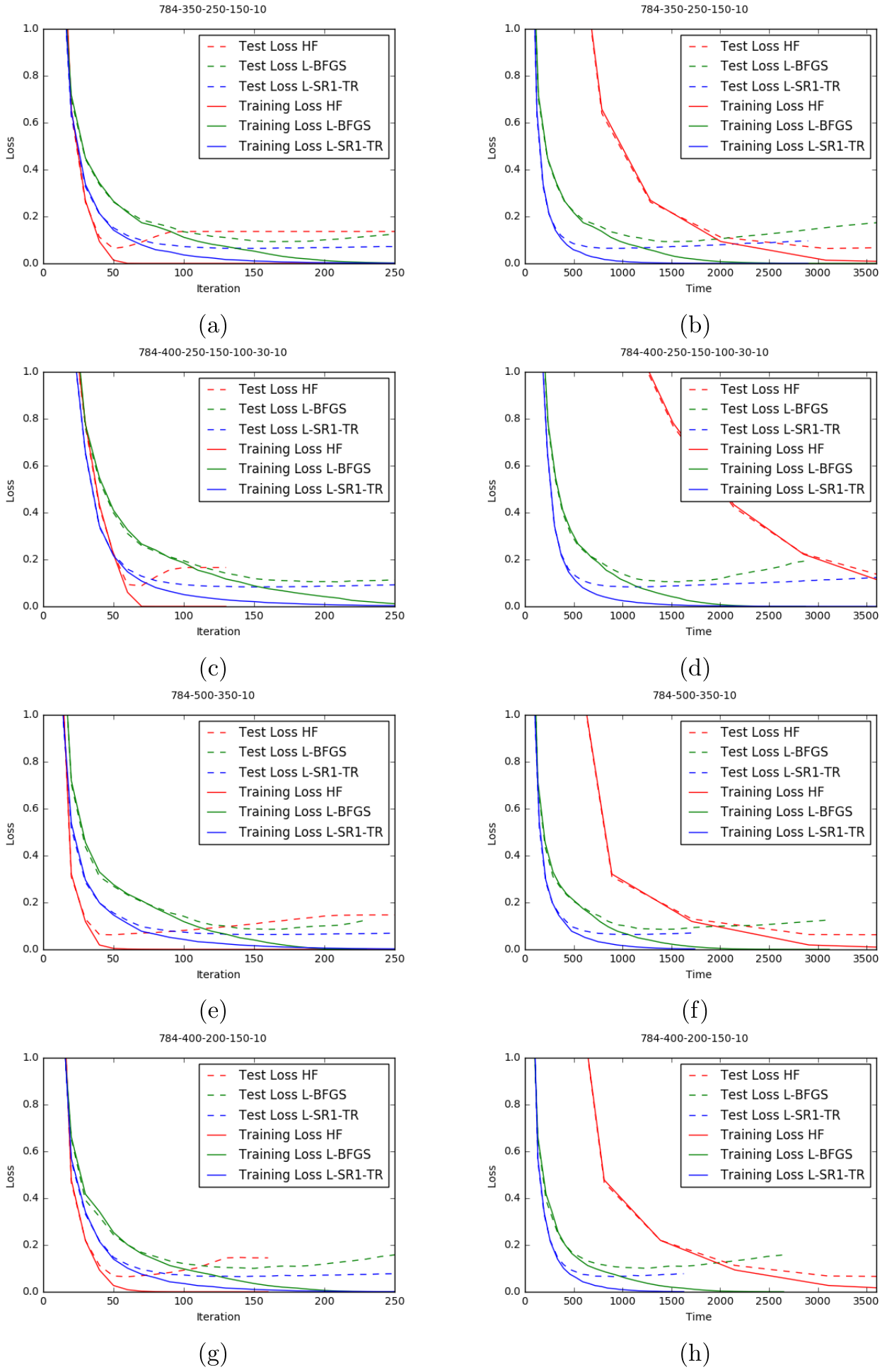}
\caption{\textbf{Experiment IA.} Plots of the loss function versus iterations (left) and  
time (right) for 
the Hessian-Free (HF), 
Limited-Memory BFGS (L-BFGS), and the proposed  
Limited-Memory SR1 Trust-Region (L-SR1-TR{})
methods
    on the MNIST data set of handwritten digits using four different sets of
    hidden layers 
    \{\{350,250,150\},
    \{400,250,150,150,100,30\},
    \{500,350\},
    \{400,200,150\}\}
    with
    input layer of size 784 and output layer of size 10.}\label{fig:MNIST1}
\end{figure}

\begin{figure}[htbp!]
\centering
\includegraphics[width=14cm]{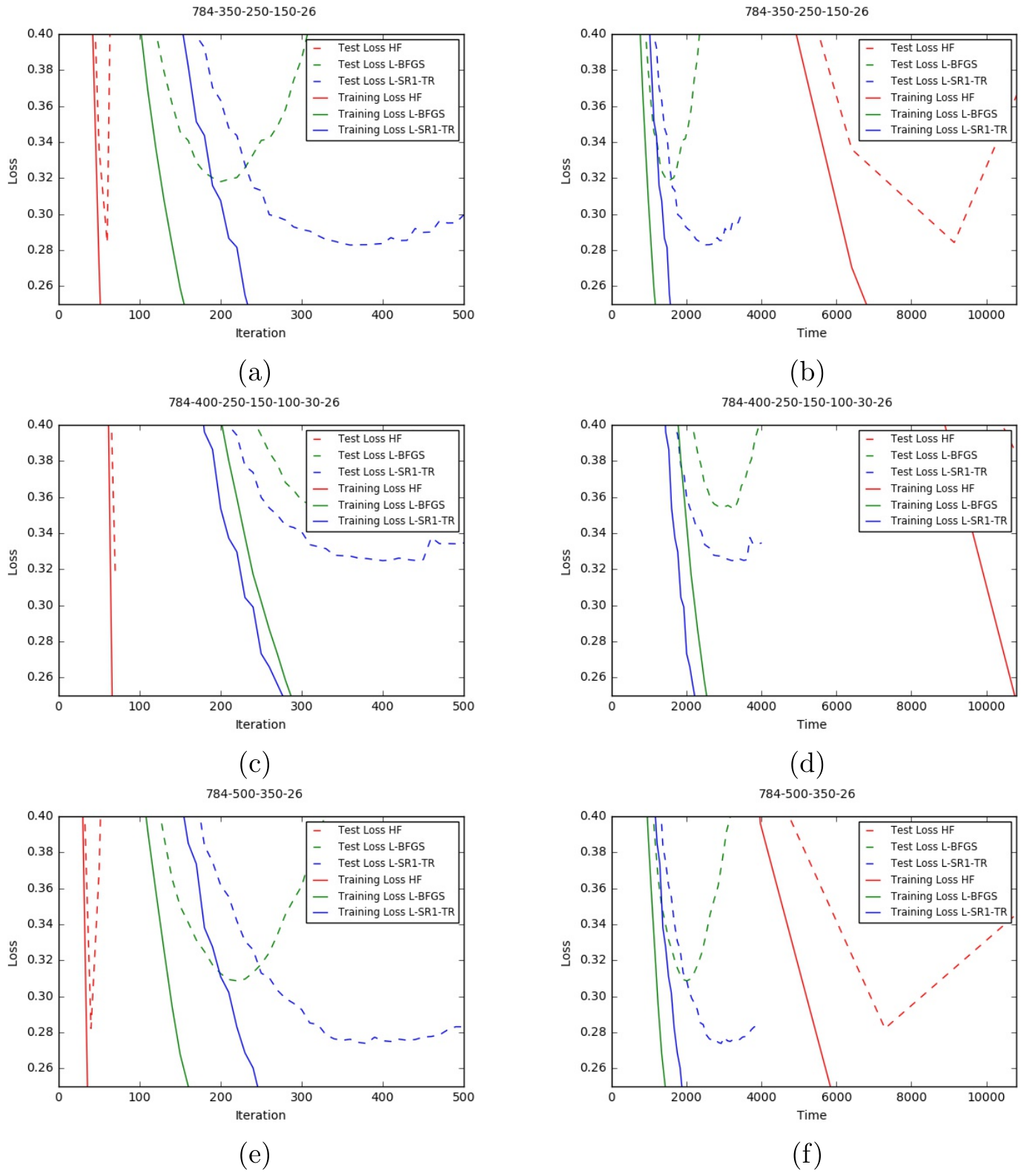}
\caption{\textbf{Experiment IB.} Plots of the loss function versus iterations (left) and  
time (right) for 
the Hessian-Free (HF), 
Limited-Memory BFGS{} (L-BFGS), and the proposed  
Limited-Memory SR1 Trust-Region (L-SR1-TR{})
methods
    on the EMNIST data set of handwritten letters using three different sets of
    hidden layers 
    \{\{350,250,150\},
    \{400,250,150,100,30\},
    \{500,350\}\}
    with
    input layer of size 784 and output layer of size 26.}\label{fig:EMNIST1}
\end{figure}


\begin{figure}[htbp!]
\centering
\includegraphics[width=14cm]{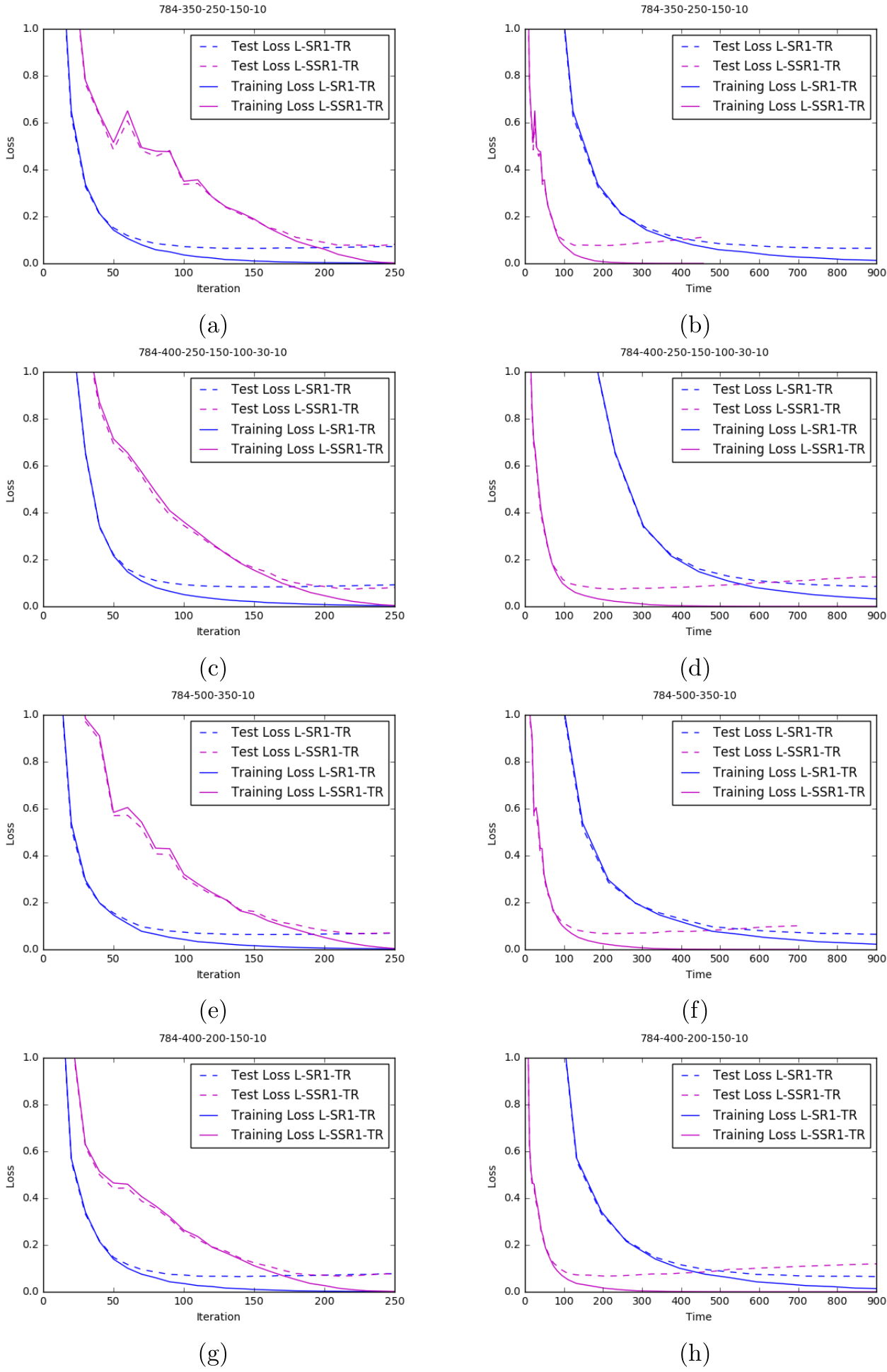}
\caption{\textbf{Experiment IIA.} Plots of the loss function versus iterations (left) and  
time (right) for the proposed   
Limited-Memory SR1 Trust-Region (L-SR1-TR{})
and Limited-Memory Stochastic SR1 Trust-Region (L-SSR1-TR{}) methods
    on the MNIST data set of handwritten digits using four different sets of 
    hidden layers 
    \{\{350,250,150\},
    \{400,250,150,150,100,30\},
    \{500,350\},
    \{400,200,150\}\}
    with
    input layer of size 784 and output layer of size 10.}\label{fig:MNIST2}
\end{figure}

\begin{figure}[htbp!]
\centering
\includegraphics[width=14cm]{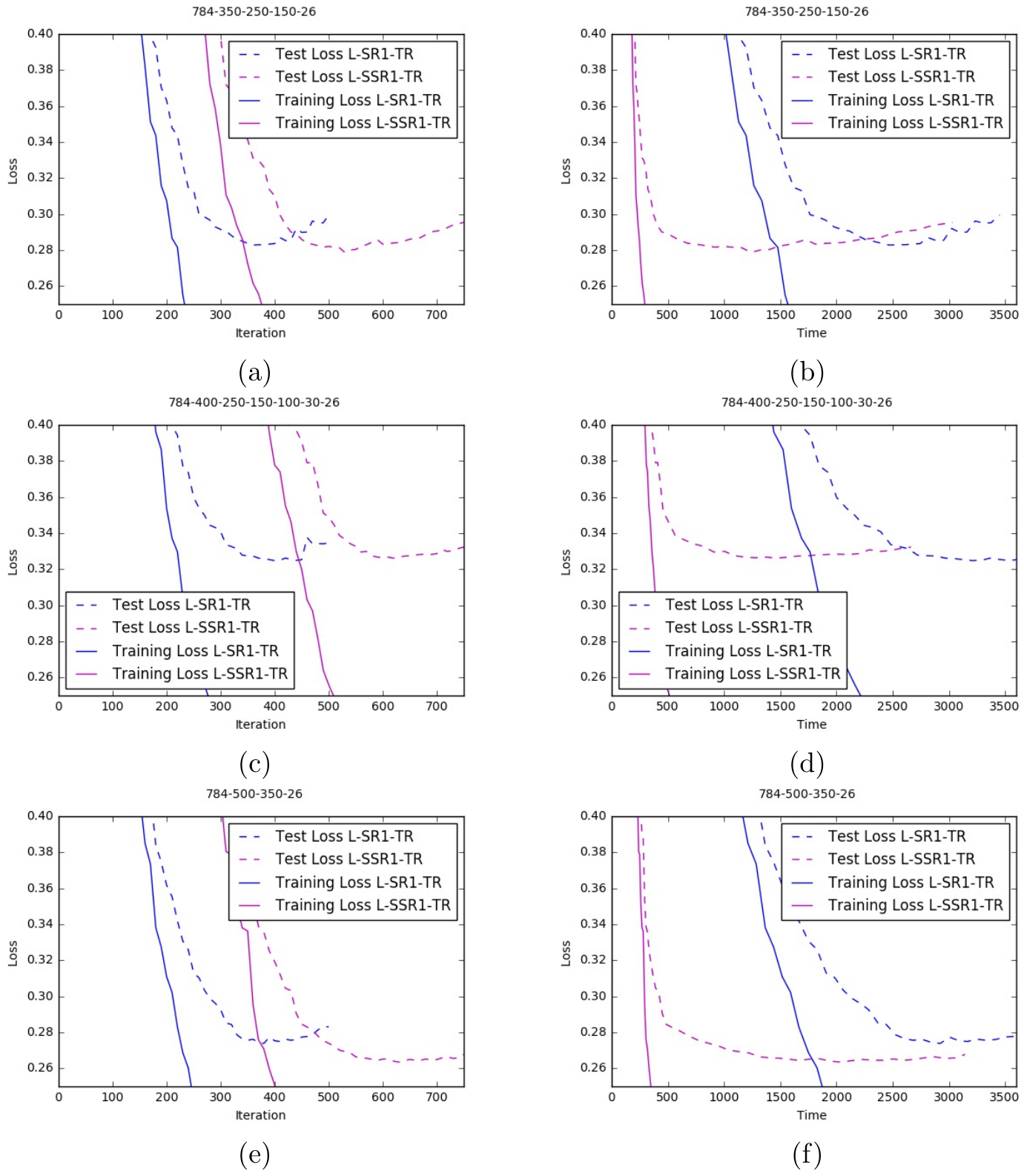}
\caption{\textbf{Experiment IIB.} Plots of the loss function versus iterations (left) and  
time (right) for the proposed   
Limited-Memory SR1 Trust-Region (L-SR1-TR{})
and Limited-Memory Stochastic SR1 Trust-Region (L-SSR1-TR{}) methods
    on the EMNIST data set of handwritten letters using four different sets of 
    hidden layers 
    \{\{350,250,150\},
    \{400,250,150,150,100,30\},
    \{500,350\}\}
    with
    input layer of size 784 and output layer of size 26.}\label{fig:EMNIST2}
\end{figure}

\end{document}